\documentclass[11pt]{amsart}

\usepackage{amssymb}
\usepackage{amsmath}
\usepackage{amsfonts}
\usepackage{graphicx}
\usepackage{amsthm}
\usepackage{enumerate}
\usepackage[mathscr]{eucal}
\usepackage{mathrsfs}
\usepackage{verbatim}

\makeatletter
\@namedef{subjclassname@2010}{%
  \textup{2010} Mathematics Subject Classification}
\makeatother

\numberwithin{equation}{section}

\theoremstyle{plain}
\newtheorem{theorem}{Theorem}[section]
\newtheorem{lemma}[theorem]{Lemma}
\newtheorem{proposition}[theorem]{Proposition}

\theoremstyle{plain}

\numberwithin{equation}{section}

\theoremstyle{remark}

\DeclareMathOperator{\supp}{supp}

\DeclareMathOperator{\diag}{diag}

\begin{document}
\date{}

\title
[An optimal stretching problem]{Lattice points in stretched model domains of finite type in $\mathbb{R}^d$}

\author{Jingwei Guo}
\address{School of Mathematical Sciences\\
University of Science and Technology of China\\
Hefei, 230026\\ P.R. China\\} \email{jwguo@ustc.edu.cn}
\author{Weiwei Wang}
\address{School of Mathematical Sciences\\
University of Science and Technology of China\\
Hefei, 230026\\ P.R. China\\} \email{wawnwg123@163.com}

\thanks{J.G. is  partially supported by the NSFC Grant No. 11571131 and 11501535.}

\begin{abstract}
We study an optimal stretching problem for certain convex domain in $\mathbb{R}^d$ ($d\geq 3$) whose boundary has points
of vanishing Gaussian curvature. We prove that the optimal domain which contains the most positive (or least nonnegative) lattice points is asymptotically balanced. This type of problem has its origin in the ``eigenvalue optimization among rectangles'' problem in spectral geometry. Our proof relies on two-term bounds for lattice counting for general convex domains in $\mathbb{R}^d$ and an explicit estimate of the Fourier transform of the characteristic function associated with the specific domain under consideration.
\end{abstract}



\maketitle

\section{Introduction}\label{introduction}

Let $\Omega \subset \mathbb R^d$ be an Euclidean domain with sufficiently smooth boundary. A coordinate \emph{stretch} of $\Omega$ is a domain of the form
\begin{equation}
\label{StretchOmega}
A\Omega := \{(a_1x_1, \ldots, a_d x_d)\ |\ (x_1, \ldots, x_d) \in \Omega\},
\end{equation}
where $a_1, \ldots, a_d$ are positive numbers, and $A=\diag(a_1, \ldots, a_d)$ is the associated $d \times d$  diagonal matrix. Such a stretch is called a volume-preserving stretch if the stretching factors $a_1, \ldots, a_d$ satisfy $a_1 \cdots a_d=1$, or in other words, $\det A=1$. In what follows we will always assume that the stretch $A\Omega$ is a volume-preserving stretch, and we will denote
\begin{equation}
a_* =\|A^{-1}\|_\infty =\max\{a_1^{-1}, \ldots, a_d^{-1}\}.\label{aStar}
\end{equation}
Since $|A\Omega|=|\Omega|$, by a simple geometric argument dated back to Gauss, it is easily seen that the number of lattice points in the enlarged domain, $tA\Omega$, is equal to $|\Omega|t^d + O(t^{d-1})$ as $t \to \infty$. Note that although the leading term in this asymptotic formula  is independent of the \emph{stretching factor} $A$, the remainder term does depend on the shape of the domain and thus depends on the stretching factor $A$.

The problem of estimating the \emph{growth rate} of the remainder term
\[\#(\mathbb Z^d \cap t\Omega)-|\Omega|t^d\]
as $t \to \infty$ is now widely known as the \emph{lattice point problem}, and has been studied extensively for over one hundred years. This is a highly nontrivial problem. Even for the unit disk in the plane, in which case the problem is the famous \emph{Gauss circle problem}, we are still far away from the conjectured optimal growth rate. For historical results on the lattice point problem, see for example \cite{survey2004}, \cite{nowak14survey}, etc.

Motivated by the ``eigenvalue optimization among rectangles'' problem in spectral geometry, in the paper \cite{AF12} P. Antunes and P. Freitas studied  the following variation of the Gauss circle problem: find the ``optimal stretching factor" that maximizes the remainder term
\[
\#(\mathbb{N}^2 \cap tA\mathcal{B}) - \pi t^2/4,
\]
(\footnote{In this paper we use $\mathbb{N}=\{1, 2, \ldots\}$, $\mathbb{Z}_{+}=\{0, 1, 2, \ldots\}$, $\mathbb{Z}_{*}^{d}=\mathbb{Z}^{d}\setminus\{0\}$, and $\mathbb{R}_{+}=[0, \infty)$.})where $\mathcal{B}=B(0, 1)$ is the unit disk in $\mathbb R^2$, and $A=\diag(s, s^{-1})$. Note that they only count \emph{positive} lattice points, i.e. integer points in the first quadrant. This is because they are considering corresponding Laplace eigenvalues with Dirichlet boundary condition. Although there is no explicit way to determine the ``best" stretching factor $s$ (not necessarily unique) for each $t$, they showed that as $t \to \infty$, the ``best factor" $s=s(t) \to 1$. In other words, among all ellipses of the same area, those that enclose the most lattice points in the first quadrant must be more and more ``round'', as the area goes to infinity.

In the last couple of years, many people started to study the above type of optimal stretching problem (for ellipses or ellipsoids) and the associated shape optimization problem for eigenvalues. See for example \cite{BuG}, \cite{BeG}, \cite{GL}, etc.

Recently, in a pair of papers R. Laugesen and S. Liu~\cite{LL16} and S. Ariturk and R. Laugesen~\cite{AL} extended P. Antunes and P. Freitas' results to very general planar domains including $p$-ellipses for $0<p<\infty$, $p\ne 1$. They showed, among others, that under mild assumptions on a given domain's boundary the most ``balanced'' domain will enclose the most positive (or least nonnegative) lattice points in the limit, where a domain in $\mathbb{R}^d$ ($d\geq 2$) (\footnote{For the sake of later usage, here we give the definition for any dimension $d\geq 2$.}) is said to be \emph{balanced} if the $(d-1)$-dimensional measures of the intersections of the domain with each coordinate hyperplane are equal. Furthermore they provided rates of convergence for optimal stretching factors.

Further generalizations to high dimensions have been investigated by N. Marshall \cite{marshall},  in which they focus on any bounded convex domain $\Omega\subset \mathbb{R}^d$ whose boundary $\partial \Omega$ is $C^{d+2}$ and has nowhere vanishing Gaussian curvature. (In Proposition \ref{2term-prop} below we will prove two-term bounds for lattice counting for general strictly convex domains, by using which one can recover N. Marshall's main theorem in \cite{marshall}.)


The class of domains studied in \cite{marshall} is quite general, though it does not include some interesting domains having boundary points of vanishing Gaussian curvature, e.g. the super sphere
\begin{equation*}
\{x\in \mathbb{R}^d :
|x_1|^{\omega}+|x_2|^{\omega}+\cdots+|x_d|^{\omega}\leq  1\},\quad \textrm{$\omega\geq 3$},
\end{equation*}
a typical domain of finite type (in the sense of J. Bruna, A. Nagel, and S. Wainger~\cite{BNW}) in $\mathbb{R}^d$. In fact, even the classical lattice point problem for domains having boundary points of vanishing Gaussian curvature is not well understood, especially in high dimensions. For partial results we refer interested readers to  \cite{survey2004}, \cite{nowak14survey}, \cite{guo1}, \cite{guo2}, and the references given there.

We remark that if the boundary is ``too flat'' the optimal stretching problem could be very subtle. For example, the case of triangles has been analyzed  by N. Marshall and S. Steinerberger in \cite{MS}, where they showed that the optimal shape needs not to be asymptotically balanced.

So a natural question is to study the optimal stretching problem for domains of finite type. In this paper we will explore this problem for the following model domain of finite
type in $\mathbb{R}^d$ ($d\geq 3$)
\begin{equation}
\mathcal{D}=\{x\in \mathbb{R}^d : x_1^{\omega_1}+\cdots+x_d^{\omega_d}\le1\}\label{domainD}
\end{equation}
with $\omega_1,\ldots,\omega_d\in 2\mathbb{N}$. For future purpose we will denote
\[\omega=\max(\omega_1, \ldots, \omega_d).\] For each $1 \le j \le d$, we let $\mathcal{D}_j$ be the intersection of $\mathcal{D}$ with the coordinate hyperplane defined by $x_j=0$.
For each $t\geq 1$ we let
\begin{equation}
A(t)=\mathrm{argmax}_A\#(\mathbb{N}^d\cap tA\mathcal{D}),\label{def-At}
\end{equation}
where the argmax  ranges over all positive definite diagonal matrices $A$ of determinant 1. We recall that the notation  $\mathrm{argmax}_x f(x)$ means the values of $x$ for which $f(x)$ attains the function's largest value. So in our setting, $A(t)$ means the stretching matrix $A$ that makes the set $tA\mathcal D$ contain most positive lattice points. Note that in general such $A$ is not unique, i.e. there could be a set of ``optimal factors''. So in what follows, when we write $A(t)=\diag(a_1(t),a_2(t),\ldots,a_d(t))$, we really mean that $A(t)$ is any stretching factor that maximizes the lattice counting function $\#(\mathbb{N}^d\cap tA\mathcal{D})$.

Our main theorem is

\begin{theorem}\label{main-thm}
Let $\mathcal{D}$ be the domain defined by \eqref{domainD}. Then any maximal stretching factor $A(t)=\diag(a_1(t),a_2(t),\ldots,a_d(t))$ in \eqref{def-At} satisfies
\begin{equation}
\left|a_j(t)-\frac{|\mathcal{D}_j|}{\sqrt[\leftroot{-3}\uproot{3}d]{|\mathcal{D}_1||\mathcal{D}_2|\cdots|\mathcal{D}_d|}}\right|=O(t^{-\gamma}), \qquad 1 \le j \le d, \label{conv-rate}
\end{equation}
where $|\mathcal D_j|$ is the $(d-1)$-dimensional measure of $\mathcal D_j$ whose formula is given in Appendix \ref{volume}, and $\gamma=\min\{(d-1)/(2\omega),(d-1)/(2d+2)\}$.
\end{theorem}

Similarly we consider the optimal stretching problem for nonnegative lattice points. For each $t\geq 1$ we let
\begin{equation}
\widetilde{A}(t)=\mathrm{argmin}_{\widetilde A}\#(\mathbb{Z}_+^d\cap t\widetilde A\mathcal{D}),\label{def-At222}
\end{equation}
where the argmin  ranges over all positive definite diagonal matrices $\widetilde A$ of determinant 1. With the same understanding as above,  when we write $\widetilde{A}(t)=\diag(\tilde a_1(t), \tilde a_2(t),\ldots, \tilde a_d(t))$, we really mean that $\widetilde{A}(t)$ is any stretching factor that minimizes the lattice counting function $\#(\mathbb{Z}_+^d\cap t\widetilde A\mathcal{D})$. Then we have

\begin{theorem}\label{min-thm}
Let $\mathcal{D}$ be the domain defined by \eqref{domainD}. Then any minimal stretching factor $\widetilde{A}(t)=\diag(\tilde a_1(t), \tilde a_2(t),\ldots, \tilde a_d(t))$ in \eqref{def-At222} satisfies
\begin{equation*}
\left|\tilde a_j(t)-\frac{|\mathcal{D}_j|}{\sqrt[\leftroot{-3}\uproot{3}d]{|\mathcal{D}_1||\mathcal{D}_2|\cdots|\mathcal{D}_d|}}\right|=O(t^{-\gamma}),  \qquad 1 \le j \le d,
\end{equation*}
where $\gamma=\min\{(d-1)/(2\omega),(d-1)/(2d+2)\}$.
\end{theorem}

Proofs of these two theorems are essentially the same. We follow the framework of P. Antunes and P. Freitas~\cite{AF12}, R. Laugesen and S. Liu~\cite{LL16}, S. Ariturk and R. Laugesen~\cite{AL}, etc. We establish, in Section \ref{2term-section}, bounds for lattice counting to control the stretching factor $A$. Then we prove, in Section \ref{mainsection}, asymptotic formulas for lattice counting to study the limiting behavior of the optimal stretching factor.

One key ingredient in the proof is the following bound (given by \cite[Theorem 2.1]{guo1}) of the Fourier transform of the characteristic
function of $\mathcal{D}$. Let $0<c\leq 1$ be a constant. For any $\xi\in S^{d-1}$ with
$|\xi_d|\geq c$ and $\lambda>0$ we have
\begin{equation}
|\widehat{\chi}_{\mathcal{D}}(\lambda \xi)|\lesssim
\lambda^{-1}\prod_{l=1}^{d-1}
\min\left(\lambda^{-\frac{1}{\omega_l}},
|\xi_l|^{-\frac{\omega_l-2}{2(\omega_l-1)}}\lambda^{-\frac{1}{2}}
\right), \label{FT-remark}
\end{equation}
(\footnote{For functions $f$ and $g$ with $g$ taking nonnegative real values,
$f\lesssim g$ means $|f|\leq Cg$ for some constant $C$. The Landau
notation $f=O(g)$ is equivalent to $f\lesssim g$. })where the implicit constant only depends on $c$ and $\mathcal{D}$.

At last we remark that interesting results on the optimal stretching problem in $\mathbb{R}^2$ with \emph{shifted} lattices can be found in R. Laugesen and S. Liu~\cite{LL17}. It would be interesting to investigate similar problems in high dimensions.

\section{Two-term bounds for lattice counting}\label{2term-section}

In this section we prove high-dimensional analogues of \cite[Proposition 6 and 9]{LL16} for strictly convex(\footnote{A domain $\Omega\subset \mathbb{R}^d$ is said to be strictly convex if the line segment connecting any two points $x$ and $y$ in $\Omega$ lies in the interior of $\Omega$, except possibly for its endpoints.}) domains in $\mathbb{R}^d$. In the proof we use the parallel section function of $\Omega$ (see A. Koldobsky~\cite[Chapter 2]{koldobsky05book}) to transfer the problem into a two-dimensional one (like N. Marshall did in \cite{marshall}) and then combine the methods used in \cite{LL16} and \cite{AL} for planar curves that are concave up or down.

We notice that instead of  proving two-term bounds N. Marshall~\cite{marshall} adopted a proof by contradiction using certain lower Riemann sums to prove his main theorem for convex domains with nonvanishing Gaussian curvature. By using Proposition \ref{2term-prop} below it is not hard to recover his main theorem. One can also possibly improve the rate of convergence by using the method in \cite{guo12}.

We recall that for a positive definite diagonal matrix $A$ with determinant $1$, the number $a_*$ is defined as in (\ref{aStar}).


\begin{proposition}\label{2term-prop}
Let $\Omega\subset \mathbb{R}^d$ be strictly convex, compact, and symmetric with respect to each coordinate hyperplane with $C^{2}$ boundary. There exists a positive constant $c$ depending only on $\Omega$ such that if $t/a_*\geq 1$ then
\begin{equation}
\#(\mathbb{N}^d\cap tA\Omega)\leq 2^{-d}|\Omega|t^d-ca_*t^{d-1}\label{ccc}
\end{equation}
and
\begin{equation}
\#(\mathbb{Z}^d_{+}\cap tA\Omega)\geq 2^{-d}|\Omega|t^d+ca_*t^{d-1}.\label{ccclowwer}
\end{equation}
\end{proposition}

\begin{proof}
Without loss of generality we assume $a_*=a_1^{-1}$. We will prove \eqref{ccc} below while
the proof of \eqref{ccclowwer} is essentially the same. We have that
\begin{equation*}
\#(\mathbb{N}^d\cap tA\Omega)=\sum_{k_1\in \mathbb{N}} \sum_{k'\in\mathbb{N}^{d-1}}\chi_{\Omega}(\frac{k_1}{a_1 t}, \ldots, \frac{k_d}{a_d t})
\end{equation*}
and, by the convexity of $\Omega$, that
\begin{equation*}
\sum_{k'\in\mathbb{N}^{d-1}}\chi_{\Omega}(\frac{k_1}{a_1 t}, \ldots, \frac{k_d}{a_d t})\le \int_{\mathbb{R}_{+}^{d-1}}\chi_{\Omega}(\frac{k_1}{a_1 t}, \frac{x_2}{a_2 t}, \ldots, \frac{x_d}{a_d t}) \,\textrm{d}x',
\end{equation*}
where $x'=(x_2, \ldots, x_d)$. Then, by change of variables and $\det(A)=1$, we get that
\begin{equation}
\#(\mathbb{N}^d\cap tA\Omega)\leq a_1^{-1}t^{d-1} \sum_{k_1\in \mathbb{N}} f(\frac{k_1}{a_1 t}),\label{bbb}
\end{equation}
where
\begin{equation}
f(x_1)=\int_{\mathbb{R}_{+}^{d-1}}\chi_{\Omega}(x_1, x') \,\textrm{d}x'\label{fff}
\end{equation}
is (up to a constant multiple) the parallel section function of $\Omega$ in the direction of $x_1$-axis, which is continuous from $\mathbb{R}_{+}$ to $\mathbb{R}_{+}$, strictly decreasing and supported on $[0, X_1]$ for some constant $X_1>0$, and bounded above by an absolute constant $f(0)$.

We claim that there exists a positive constant $c$ depending only on $\Omega$ such that
\begin{equation}
\int_0^{\infty} f(\frac{x_1}{a_1 t})\,\textrm{d}x_1-\sum_{k_1\in \mathbb{N}} f(\frac{k_1}{a_1 t})\geq c. \label{aaa}
\end{equation}
Since the integral in \eqref{aaa} is equal to $2^{-d} |\Omega| a_1 t$, the desired estimate \eqref{ccc} follows easily from \eqref{bbb} and \eqref{aaa}.

If we denote
\begin{equation*}
F(x_1)=f(x_1/a_1t),
\end{equation*}
to prove the claim it suffices to show that the area between the curves of $F$ and a step function $\sum_{j\in\mathbb{N}}F(j)\chi_{[j-1, j)}(x_1)$ has a lower bound $c$ (depending only on $\Omega$), by comparing it with the total area of certain ``inbetween'' triangles.

Notice that there must exist an interval $[p-2\delta, p+2\delta]\subset (0, X_1)$ with $\delta>0$, on which $f''$ is always nonnegative or nonpositive. Indeed, since $f$ is two times continuously differentiable in some neighborhood of zero (see \cite[Lemma 2.4]{koldobsky05book}), if there exists a point $p$ in that neighborhood such that $f''(p)\ne 0$ then the continuity of $f''$ ensures the existence of such an interval.

We first assume that $a_1 t\geq 2/\delta$ and study $F$  mainly on the interval $[(p-\delta)a_1 t, (p+\delta) a_1 t]$.

If $F$ is always nonpositive hence concave down we study the ``inbetween'' right triangles of width $1$ with vertices $(i, F(i+1))$, $(i+1, F(i+1))$, and $(i, F(i))$, with $\lfloor(p-\delta)a_1 t\rfloor+1\leq i\leq \lfloor(p+\delta)a_1 t\rfloor-1$. The total area of these triangles is equal to
\begin{align}
\sum_{i=\lfloor(p-\delta)a_1 t\rfloor+1}^{\lfloor(p+\delta)a_1 t\rfloor-1}&\frac{1}{2}\left(F(i)-F(i+1)\right)\nonumber \\
=&\frac{1}{2}\left(F(\lfloor(p-\delta)a_1 t\rfloor+1)-F(\lfloor(p+\delta)a_1 t\rfloor)\right)\nonumber\\
\geq &\frac{1}{2}\left(f\left(p-\frac{\delta}{2}\right)-f\left(p+\frac{\delta}{2}\right)\right),\label{const1}
\end{align}
where in the last inequality we have used $a_1t\geq 2/\delta$.

If $F$ is always nonnegative hence concave up we study the ``inbetween'' right triangles of width $1$ with vertices $(i, F(i+1))$, $(i+1, F(i+1))$, and $(i, F(i+1)-F'(i+1))$, where $\lfloor(p-\delta)a_1 t\rfloor\leq i\leq \lfloor(p+\delta)a_1 t\rfloor-2$. The total area of these triangles is equal to
\begin{align}
\sum_{i=\lfloor(p-\delta)a_1 t\rfloor}^{\lfloor(p+\delta)a_1 t\rfloor-2}\frac{1}{2}|F'(i+1)|&\geq \sum_{i=\lfloor(p-\delta)a_1 t\rfloor}^{\lfloor(p+\delta)a_1 t\rfloor-2}\frac{1}{2}|F(i+1)-F(i+2)|\nonumber\\
&=\frac{1}{2}\left(F(\lfloor(p-\delta)a_1 t\rfloor+1)-F(\lfloor(p+\delta)a_1 t\rfloor)\right)\nonumber,
\end{align}
which has the same lower bound as in \eqref{const1}.

If $1\leq a_1 t\leq 2/\delta$, we observe that the left hand side of \eqref{aaa} is continuous in $a_1t$ and always positive (since $f$ is strictly decreasing), therefore it must have a uniform lower bound $c_1>0$.

By choosing $c$ sufficiently small (say, smaller than $c_1$ and the absolute constants in \eqref{const1}) , we get \eqref{aaa}. This finishes the proof of \eqref{ccc}.
\end{proof}


\section{Asymptotics for lattice counting}\label{mainsection}

In this section we prove several asymptotic formulas for lattice counting. See N. Marshall~\cite{marshall} for the case when the domain's boundary has nonvanishing Gaussian curvature.

We first prove an asymptotic formula for the number of lattice points in the domain $tA\mathcal{D}$.
Its proof is very standard except that we need to take care of the effect of the transformation $A$.

\begin{proposition}\label{latticeInWhole}
Let $\mathcal{D}$  be the domain  defined by \eqref{domainD}. For any positive definite diagonal matrix $A=\diag(a_1,\ldots,a_d)$ of determinant $1$, if $t/a_*\geq 1$ then
\begin{equation}
\#(\mathbb{Z}^d\cap tA\mathcal{D})=|\mathcal{D}|t^d+O\left(a_*^{1+\frac{d-1}{\omega}} t^{(d-1)(1-\frac{1}{\omega})}+a_*^{2-\frac{2}{d+1}}t^{d-2+\frac{2}{d+1}}\right),\label{ggg}
\end{equation}
where the implicit constant depends only on the domain $\mathcal{D}$.
\end{proposition}

\begin{proof}
Let $0\leq \rho\in C_0^{\infty}(\mathbb{R}^d)$ be a bump function such that
$\supp \rho\subset B(0, 1)$, $\int_{\mathbb{R}^d}\rho(x) \, \textrm{d}x=1$, $\rho_{\varepsilon}(x)= \varepsilon^{-d}\rho(\varepsilon^{-1}x)$, $0<\varepsilon<1$,
and
\begin{equation*}
N_{A,\varepsilon}(t)=\sum_{k\in\mathbb{Z}^d}\chi_{tA
\mathcal{D}}\ast\rho_{\varepsilon}(k),
\end{equation*}
where $\chi_{tA\mathcal{D}}$ denotes the characteristic function of $tA\mathcal{D}$. By the Poisson summation formula, we have
\begin{equation}
N_{A,\varepsilon}(t)=t^d\sum_{k\in\mathbb{Z}^d}\widehat{\chi}_{\mathcal{D}}(tAk)\widehat{\rho}(\varepsilon k)
=|\mathcal{D}|t^d+R_{A,\varepsilon}(t),\label{hhh}
\end{equation}
where
\begin{equation*}
R_{A,\varepsilon}(t)=t^d\sum_{k\in\mathbb{Z}_\ast^d}\widehat{\chi}_{\mathcal{D}}(tAk)\widehat{\rho}(\varepsilon k).
\end{equation*}
By adding a partition of unity $\sum_{j=1}^{d}\Omega_j\equiv1$ with each $\Omega_j$ supported in $\Gamma_j=\{x\in\mathbb{R}^d:|x_j|\ge (2d)^{1/2}|x|\}$ and smooth away from the origin, we have
\begin{equation*}
R_{A,\varepsilon}(t)=\sum_{1\le i,j\le d}S_{ij}
\end{equation*}
with
\begin{equation*}
S_{ij}=t^d\sum_{(i)}\Omega_j(Ak)\widehat{\chi}_{\mathcal{D}}(tAk)\widehat{\rho}(\varepsilon k),
\end{equation*}
where the summation is over all lattice points $k$'s with exactly $1\leq i\leq d$ nonzero components such that $Ak\in \supp (\Omega_j)$.

We may assume $j=d$ below since other cases are similar due to symmetry.

If $i=1$, by \cite[Theorem 2.1]{guo1} (see \eqref{FT-remark}) we have
\begin{equation}
|S_{1d}|\lesssim a_d^{-1-\sum_{l=1}^{d-1}\frac{1}{\omega_l}}t^{d-1-\sum_{l=1}^{d-1}\frac{1}{\omega_l}}\leq a_*^{1+\sum_{l=1}^{d-1}\frac{1}{\omega_l}}t^{d-1-\sum_{l=1}^{d-1}\frac{1}{\omega_l}},\label{s1d}
\end{equation}
where in the last inequality we have used the definition of $a_*$.

If $2\le i\le d$, by using \cite[Theorem 2.1]{guo1}, $|a_l|\geq a_*^{-1}$, and $|Ak|\geq a_*^{-1}|k|$, and comparing the sums with integrals in polar coordinates, we get
\begin{equation}
|S_{id}|\lesssim \sum_{S\in \mathcal{P}_i(\mathbb{N}_d) : d\in S} a_*^{\frac{i+1}{2}+\alpha(S)}t^{d-\frac{i+1}{2}-\alpha(S)}\left(1+\varepsilon^{-\frac{i-1}{2}+\alpha(S)}\right),\label{sid}
\end{equation}
where $\mathbb{N}_d=\{1, 2, \ldots, d\}$,  $\mathcal{P}_i(\mathbb{N}_d)$ is the collection of all subsets of
$\mathbb{N}_d$ having $i$ elements, and
\begin{equation*}
\alpha(S):=\sum\limits_{1\leq l\leq d, \, l\notin S}\frac{1}{\omega_l}.
\end{equation*}
Note that the contribution coming from the term $1$ on the right side of \eqref{sid} is not larger than
the right side of \eqref{s1d} if $a_*<t$.

By \eqref{s1d}, \eqref{sid}, and similar bounds for other $j$'s we obtain that
\begin{align}
\begin{split}
|R_{A,\varepsilon}(t)|&\lesssim \sum_{S\in \mathcal{P}_1(\mathbb{N}_d)} a_*^{1+\alpha(S)} t^{d-1-\alpha(S)}\\
        &\quad +\sum_{S\in P_i(\mathbb{N}_d),2\le i\le d}a_*^{\frac{i+1}{2}+\alpha(S)}t^{d-\frac{i+1}{2}-\alpha(S)}\varepsilon^{-\frac{i-1}{2}+\alpha(S)}
\end{split}\label{jjj}
\end{align}
whenever $a_*<t$.

By M\"uller \cite[Lemma 3]{mullerI} we have
\begin{equation}
N_{A,\varepsilon}(t-a_*\varepsilon)\le   \#(\mathbb{Z}^d\cap tA\mathcal{D})  \le N_{A,\varepsilon}(t+a_*\varepsilon).\label{kkk}
\end{equation}
Let
\begin{equation*}
\varepsilon=\left(t/a_*\right)^{-\frac{d-1}{d+1}}.
\end{equation*}
Note that $t/2\leq t\pm a_*\varepsilon \leq 3t/2$ if $t/a_*\geq C$ for a sufficiently large constant $C$. Combining \eqref{hhh}, \eqref{jjj}, and \eqref{kkk} yields
\begin{align*}
\left|\#(\mathbb{Z}^d\cap tA\mathcal{D})-|\mathcal{D}|t^d\right|&\lesssim a_*t^{d-1}\varepsilon+|R_{A,\varepsilon}(t\pm a_*\varepsilon)|\\
&\lesssim a_*^{1+\frac{d-1}{\omega}} t^{(d-1)(1-\frac{1}{\omega})}+a_*^{2-\frac{2}{d+1}}t^{d-2+\frac{2}{d+1}},
\end{align*}
where the last inequality can be verified by a direct computation. Hence we get \eqref{ggg} when $t/a_*\geq C$.

If $1\leq t/a_*\leq C$ applying the previous argument to $\#(\mathbb{Z}^d\cap (Ct)A\mathcal{D})$ yields that
\begin{equation*}
 \#(\mathbb{Z}^d\cap (Ct)A\mathcal{D})\lesssim t^d.
\end{equation*}
Hence
\begin{equation*}
\#(\mathbb{Z}^d\cap tA\mathcal{D})\lesssim t^d,
\end{equation*}
which trivially leads to the desired asymptotic formula \eqref{ggg}.
\end{proof}


Recall that  $\mathcal{D}_j\subset \mathbb{R}^d$ is the intersection of $\mathcal{D}$ (defined by \eqref{domainD}) with the coordinate hyperplane defined by $x_j=0$. As a consequence of Proposition \ref{latticeInWhole} we get the number of lattice points in the dilated and stretched intersections.

\begin{proposition}\label{latticege0}
Let $A=\diag(a_1,\ldots,a_d)$ be a positive definite diagonal matrix of determinant $1$, if $t/a_*\geq 1$ then
\begin{equation}
\begin{split}
&\#(\mathbb{Z}^d\cap tA(\cup_{j=1}^d\mathcal{D}_j))-\sum_{j=1}^da_j^{-1}|\mathcal{D}_j|t^{d-1}\\
&\quad=O \left(  a_*^{1+\frac{d-1}{\omega}} t^{(d-1)(1-\frac{1}{\omega})}+a_*^{2-\frac{2}{d+1}}t^{d-2+\frac{2}{d+1}}\right),
\end{split}\label{unionforhyper}
\end{equation}
where the implicit constant depends only on the domain $\mathcal{D}$.
\end{proposition}


\begin{proof}
For $1\leq j\leq d$ we let $A_j=\diag(a_1, \ldots,a_{j-1},a_{j+1},\ldots, a_d)$ be a $(d-1)\times (d-1)$ matrix and use the notation $\mathcal{D}_j$ to represent the set
\begin{equation*}
\{(x_1,\ldots,x_{j-1},x_{j+1},\ldots,x_d)\in \mathbb{R}^{d-1} : \sum_{1\leq l\leq d, l\neq j} x_l^{\omega_l}\leq 1\}
\end{equation*}
as well. The precise meaning of $\mathcal{D}_j$ will be clear from the context and this abuse of notation will not cause any problem. Note that
\begin{equation*}
\#(\mathbb{Z}^{d-1}\cap tA_j\mathcal{D}_j)=\#\left(\mathbb{Z}^{d-1}\cap \left(t a_j^{-\frac{1}{d-1}}\right)\left(a_j^{\frac{1}{d-1}}A_j\right)\mathcal{D}_j\right).
\end{equation*}
Applying Proposition \ref{latticeInWhole} to the right hand side of the equation above yields
\begin{equation}
\begin{split}
&\left|\#(\mathbb{Z}^{d-1}\cap tA_j\mathcal{D}_j)-a_j^{-1}|\mathcal{D}_j|t^{d-1}\right|\\
&\quad \lesssim a_j^{-1}\left(a_*^{1+\frac{d-2}{\omega}}t^{(d-2)(1-\frac{1}{\omega})}+a_*^{2-\frac{2}{d}}t^{d-3+\frac{2}{d}}\right).
\end{split}\label{dimd1}
\end{equation}

For $1\leq j\ne k\leq d$ let $A_{j,k}$ be the  $(d-2)\times (d-2)$ diagonal matrix  by removing the $j$-th and $k$-th columns and rows from $A$, and $\mathcal{D}_{j,k}=\mathcal{D}_j\cap\mathcal{D}_k\subset \mathbb{R}^d$. As above we abuse the notation $\mathcal{D}_{j,k}$ for a subset of $\mathbb{R}^{d-2}$. Since
\begin{equation*}
 \#(\mathbb{Z}^{d-2}\cap tA_{j,k}\mathcal{D}_{j,k})=\#\left(\mathbb{Z}^{d-2}\cap \left(t(a_ja_k)^{-\frac{1}{d-2}}\right)\left((a_ja_k)^{\frac{1}{d-2}}A_j\right)\mathcal{D}_{j,k}\right),
 \end{equation*}
applying Proposition \ref{latticeInWhole} again gives
\begin{equation}
\begin{split}
&\left|\#(\mathbb{Z}^{d-2}\cap tA_{j,k}\mathcal{D}_{j,k})-(a_ja_k)^{-1}|\mathcal{D}_{j,k}|t^{d-2}\right|\\
&\quad \lesssim (a_ja_k)^{-1}\left(a_*^{1+\frac{d-3}{\omega}}t^{(d-3)(1-\frac{1}{\omega})}+a_*^{2-\frac{2}{d-1}}t^{d-4+\frac{2}{d-1}}\right),
\end{split}\label{dimd2}
\end{equation}
where $|\mathcal{D}_{j,k}|$ denotes the $(d-2)$-dimensional measure of $\mathcal{D}_{j,k}$.

Since
\begin{align*}
&\sum_{j=1}^d\#(\mathbb{Z}^{d-1}\cap tA_j\mathcal{D}_j)-\sum_{1\leq j<k \le d}\#(\mathbb{Z}^{d-2}\cap tA_{j,k}\mathcal{D}_{j,k})\\
&\qquad\leq \#(\mathbb{Z}^d\cap tA(\cup_{j=1}^d\mathcal{D}_j))\le\sum_{j=1}^d\#(\mathbb{Z}^{d-1}\cap tA_j\mathcal{D}_j),
\end{align*}
we get \eqref{unionforhyper} by combining \eqref{dimd1}, \eqref{dimd2}, the definition of $a_*$, and $t/a_*\geq 1$.
\end{proof}

Using Proposition \ref{latticeInWhole} and \ref{latticege0}, we get the numbers of positive and nonnegative lattice points in $tA\mathcal{D}$.


\begin{theorem}\label{latticeInpositive}
Let $A=\diag(a_1,\ldots,a_d)$ be a positive definite diagonal matrix of determinant $1$, if $t/a_*\geq 1$ then
\begin{equation}\label{numpositive}
\begin{split}
\#(\mathbb{N}^d\cap tA\mathcal{D})&=2^{-d}|\mathcal{D}|t^d-2^{-d}\sum_{j=1}^da_j^{-1}|\mathcal{D}_j|t^{d-1}\\
      &\quad+O\left( a_*^{1+\frac{d-1}{\omega}} t^{(d-1)(1-\frac{1}{\omega})}+a_*^{2-\frac{2}{d+1}}t^{d-2+\frac{2}{d+1}}\right)
\end{split}
\end{equation}
and
\begin{equation}\label{numnonnega}
\begin{split}
\#(\mathbb{Z}^d_+\cap tA\mathcal{D})&=2^{-d}|\mathcal{D}|t^d+2^{-d}\sum_{j=1}^da_j^{-1}|\mathcal{D}_j|t^{d-1}\\
&\quad+O\left( a_*^{1+\frac{d-1}{\omega}} t^{(d-1)(1-\frac{1}{\omega})}+a_*^{2-\frac{2}{d+1}}t^{d-2+\frac{2}{d+1}}\right),
\end{split}
\end{equation}
where the implicit constants depend only on the domain $\mathcal{D}$.
\end{theorem}

\begin{proof}
By symmetry of the domain $\mathcal{D}$ we have that
\begin{equation*}
\#(\mathbb{N}^d\cap tA\mathcal{D})=2^{-d}\left(\#(\mathbb{Z}^d\cap tA\mathcal{D})-\#(\mathbb{Z}^d\cap tA(\cup_{j=1}^d\mathcal{D}_j))\right).
\end{equation*}
Then \eqref{numpositive} follows from \eqref{ggg} and \eqref{unionforhyper}.

For nonnegative lattice points, we know
\begin{equation*}
\#(\mathbb{Z}_+^d\cap tA\mathcal{D})=2^{-d}\left(\#(\mathbb{Z}^d\cap tA\mathcal{D})+\#(\mathbb{Z}^d\cap tA(\cup_{j=1}^d\mathcal{D}_j))+R_{\mathcal{D}}(A,t)\right),
\end{equation*}
where
\begin{equation}\label{errorinbdy}
R_{\mathcal{D}}(A,t)\lesssim\sum_{1\leq j<k \le d}\#(\mathbb{Z}^{d-2}\cap tA_{j,k}\mathcal{D}_{j,k})
\end{equation}
if $t/a_*\geq 1$ (see \cite{GL}). The implicit constant in \eqref{errorinbdy} depends only on the dimension of the domain $\mathcal{D}$. Thus \eqref{numnonnega} follows from \eqref{ggg}, \eqref{unionforhyper}, and \eqref{dimd2}.
\end{proof}


\section{Proofs of Main Theorems}
In this section we prove Theorem \ref{main-thm}.  The nonnegative case, i.e. Theorem \ref{min-thm}, can be handled essentially in the same way.

Let
\begin{equation*}
B=\diag\left(\frac{|\mathcal{D}_1|}{\sqrt[\leftroot{-3}\uproot{3}d]{|\mathcal{D}_1||\mathcal{D}_2|\cdots|\mathcal{D}_d|}},
\ldots,\frac{|\mathcal{D}_d|}{\sqrt[\leftroot{-3}\uproot{3}d]{|\mathcal{D}_1||\mathcal{D}_2|\cdots|\mathcal{D}_d|}}\right)
\end{equation*}
be a diagonal $d\times d$ matrix. Using \eqref{numpositive} with $A =B$, we get
\begin{equation}\label{idlowwer}
\#(\mathbb{N}^d\cap tB\mathcal{D})\geq 2^{-d}|D|t^d-Ct^{d-1}
\end{equation}
for sufficiently large $t$, where $C=2^{1-d}d(\prod_{j=1}^{d}|\mathcal{D}_j|)^{1/d}$.

For every sufficiently large $t$ let $A(t)$ be a fixed optimal stretching matrix (see \eqref{def-At}). Thus
\begin{equation}\label{lll}
\#(\mathbb{N}^d\cap tB\mathcal{D})\leq\#(\mathbb{N}^d\cap tA(t)\mathcal{D}).
\end{equation}
Denote $a_*(t)=\|A(t)^{-1}\|_\infty$. Then
\begin{equation*}
t/a_*(t)\geq 1,
\end{equation*}
otherwise $tA(t)\mathcal{D}$ does not contain any positive lattice point. Therefore Proposition \ref{2term-prop} gives
\begin{equation}\label{optimalupper}
\#(\mathbb{N}^d\cap tA(t)\mathcal{D})\leq 2^{-d}|\mathcal{D}|t^d-ca_*(t)t^{d-1}
\end{equation}
for some constant $c>0$ depending only on $\mathcal{D}$. Combining \eqref{idlowwer}, \eqref{lll}, and \eqref{optimalupper} yields
\begin{equation*}
a_*(t)\leq C/c,
\end{equation*}
which means that for sufficiently large $t$ any optimal stretching factor $A(t)$ has its $a_*$ bounded uniformly from above.

Then \eqref{numpositive} gives
\begin{align*}
\#(\mathbb{N}^d\cap tA(t)\mathcal{D})&=2^{-d}|\mathcal{D}|t^d-2^{-d}\sum_{j=1}^d a_j(t)^{-1}|\mathcal{D}_j|t^{d-1}\\
&\quad +O(t^{(d-1)(1-\frac{1}{\omega})}+t^{d-2+\frac{2}{d+1}})
\end{align*}
and
\begin{align*}
\#(\mathbb{N}^d\cap tB\mathcal{D})&=2^{-d}|\mathcal{D}|t^d-2^{-d}d\sqrt[\leftroot{-3}\uproot{3}d]{|\mathcal{D}_1||\mathcal{D}_2|\cdots|\mathcal{D}_d|}t^{d-1}\\
&\quad +O(t^{(d-1)(1-\frac{1}{\omega})}+t^{d-2+\frac{2}{d+1}}).
\end{align*}
Combining these two asymptotics with \eqref{lll} yields
\begin{equation*}
\sum_{j=1}^d a_j(t)^{-1}\frac{|\mathcal{D}_j|}{\sqrt[\leftroot{-3}\uproot{3}d]{|\mathcal{D}_1||\mathcal{D}_2|\cdots|\mathcal{D}_d|}}\leq d+O\left(t^{-\frac{d-1}{\omega}}+t^{-\frac{d-1}{d+1}}\right).
\end{equation*}
Then \eqref{conv-rate} follows easily from Lemma \ref{elem-lemma}. This finishes the proof.


\appendix

\section{Computation of the volume of $\mathcal{D}\cap\mathbb{R}^d_+$}\label{volume}

In this part we show that the volume of $\mathcal{D}\cap\mathbb{R}_+^d$ is
\begin{equation}\label{VolD}
V(\omega_1, \ldots, \omega_d):=\left(\sum_{l=1}^d \frac{\omega_1 \omega_2\cdots \omega_d}{\omega_l}\right)^{-1}
\frac{\Gamma(\frac{1}{\omega_1})\Gamma(\frac{1}{\omega_2})\cdots
\Gamma(\frac{1}{\omega_d})}{\Gamma(\sum_{l=1}^d\frac{1}{\omega_l})}.
\end{equation}
As a consequence, the $|\mathcal D_j|$'s in Theorem \ref{main-thm} are explicitly given by
$$|\mathcal D_j|=2^{d-1}V(\omega_1, \ldots, \omega_{j-1}, \omega_{j+1}, \ldots, \omega_d).$$

To prove (\ref{VolD}), we denote $V_{\mathcal{D}}=|\mathcal{D}\cap\mathbb{R}_{+}^d|$ and we change coordinates from $x_l$ to $y_l=x_l^{k_l}$, where $k_l=\omega_l/2$. Then we have
\[\aligned
V_{\mathcal{D}} & =\int_{\{x\in \mathbb{R}^d:x_1^{\omega_1}+\cdots+x_d^{\omega_d}\le1\}\cap\mathbb{R}^d_{+}}\textrm{d}x_1\cdots\textrm{d}x_d
\\ &=\int_{\{y\in \mathbb{R}^d:y_1^{2}+\cdots+y_d^{2}\le1\}\cap\mathbb{R}^d_{+}}\prod_{l=1}^d\frac{1}{k_l}y_l^{\frac{1}{k_l}-1}\textrm{d}y_1\cdots\textrm{d}y_d.
\endaligned\]
The last integral can be computed in polar coordinates $(r,\theta_1, \ldots, \theta_{d-1})$, where $r=|y|=\sqrt{y_1^2+y_2^2+\cdots+y_d^2}$ such that $y_m=r\cos\theta_m\prod_{l=1}^{m-1}\sin\theta_l$
for $1\le m\le d-1$ and $y_{d}=r\prod_{l=1}^{d-1}\sin\theta_l.$ We then get
\begin{align*}
V_{\mathcal{D}}=&\int_0^1\!\!\!\int_0^{\frac{\pi}{2}}\!\!\!\cdots\!\!\!\int_0^{\frac{\pi}{2}}\frac{1}{k_1}(r\cos \theta_1)^{\frac{1}{k_1}-1}\frac{1}{k_2}(r\sin \theta_1\cos\theta_2)^{\frac{1}{k_2}-1}\cdots
\\&\frac{1}{k_m}(r\cos\theta_m\prod_{l=1}^{m-1}\sin\theta_l)^{\frac{1}{k_m}-1}
\cdots\frac{1}{k_d}(r\prod_{l=1}^{d-2}\sin\theta_l)^{\frac{1}{k_d}-1}
\\&r^{d-1}\sin^{d-2}\theta_1\sin^{d-3}\theta_2\cdots\sin\theta_{d-2}
\textrm{d}r\textrm{d}\theta_1\cdots\textrm{d}\theta_{d-1}.
\end{align*}
A direct computation gives
\begin{align*}
V_{\mathcal{D}}=&(\sum_{l=1}^d \frac{\omega_1 \omega_2\cdots \omega_d}{\omega_l})^{-1}B(\frac{1}{2k_1},\frac 12\! \sum_{l=2}^d\!\frac{1}{k_l})B(\frac{1}{2k_2},\frac 12 \!\sum_{l=3}^d\!\frac{1}{k_l}) \cdots  B(\frac{1}{2k_{d-1}},\frac{{1}}{2k_d})
\\=&(\sum_{l=1}^d \frac{\omega_1 \omega_2\cdots \omega_d}{\omega_l})^{-1}\frac{\Gamma(\frac{1}{\omega_1})\Gamma(\frac{1}{\omega_2})\cdots\Gamma(\frac{1}{\omega_d})}{\Gamma(\sum_{l=1}^d\frac{1}{\omega_l})},
\end{align*}
where $\Gamma(x)$ is the gamma function and $B(x,y)=\Gamma(x)\Gamma(y)/\Gamma(x+y)$ is the beta function.

\section{An elementary lemma}

The following lemma is a generalization of \cite[Lemma 13]{LL16}.

\begin{lemma}\label{elem-lemma}
If $0<\varepsilon<1$ and $s_1, \ldots, s_d$ are positive numbers such that $\prod_{j=1}^d s_j=1$ and $\sum_{j=1}^d s_j\leq d+\varepsilon$, then
\begin{equation}
s_j=1+O(\varepsilon^{1/2}), \quad \textrm{$1\leq j\leq d$}.\label{elem-ineq}
\end{equation}
\end{lemma}

\begin{proof}
The assumptions imply that $(d+1)^{1-d}\leq s_j\leq d+1$ for $1\leq j\leq d$, namely, all $s_j$'s are uniformly bounded from
above and below. Denote $s_j=\exp(u_j)$. Then $\sum_{j=1}^{d} u_j=0$. Hence, by Taylor's formula, we have
\begin{equation*}
d+\varepsilon\geq \sum_{j=1}^d s_j=d+\frac{1}{2}\sum_{j=1}^de^{c_j}u_j^2,
\end{equation*}
where each $c_j$ is between $0$ and $u_j$ (which means each $\exp(c_j)$ is between $1$ and $s_j$, thus uniformly bounded from below). Therefore $u_j=O(\varepsilon^{1/2})$, which leads to \eqref{elem-ineq}.
\end{proof}


\subsection*{Acknowledgments}
We would like to express our sincere gratitude to Zuoqin Wang for valuable advices and help and to Rick Laugesen for various suggestions towards improving this exposition.



\begin{thebibliography}{99}



\bibitem{AF12}
Antunes, P. R. S., Freitas, P.: Optimal spectral rectangles and lattice ellipses. \emph{Proc. R. Soc. Lond. Ser. A  Math. Phys. Eng. Sci.}, \textbf{469} (2013), no. 2150, 20120492, 15 pp.

\bibitem{AL}
Ariturk, S., Laugesen, R. S.: Optimal stretching for lattice points under convex curves. ArXiv: 1701.03217. To appear in \emph{Portugaliae Mathematic}.



\bibitem{BuG}
van den Berg, M., Bucur, D., Gittins, K.: Maximising Neumann eigenvalues on rectangles. \emph{Bull. Lond. Math. Soc}, \textbf{48}, 877--894 (2016).

\bibitem{BeG}
van den Berg, M., Gittins, K.: Minimising Dirichlet eigenvalues on cuboids of unit measure. \emph{Mathematika}, \textbf{63}, 469--482 (2017).


\bibitem{BNW}
Bruna, J., Nagel, A., Wainger, S.: Convex hypersurfaces and Fourier
transforms. \emph{Ann. of Math. (2)}, \textbf{127}, 333--365 (1988).

\bibitem{GL}
Gittins, K., Larson, S.: Asymptotic behaviour of cuboids optimising Laplacian eigenvalues. ArXiv: 1703.10249.

\bibitem{guo12}
Guo, J.: On lattice points in large convex bodies, \emph{Acta Arith.} \textbf{151}, 83–-108 (2012).

\bibitem{guo1}
Guo, J.: A note on lattice points in model domains of finite type in $\mathbb{R}^d$,
\emph{Arch. Math. (Basel)} \textbf{108}, 45--53 (2017).


\bibitem{guo2}
Guo, J., Jiang, T.: Lattice points in model domains of finite type in $\mathbb{R}^d$, II. ArXiv: 1708.08036.


\bibitem{survey2004}
Ivi\'c, A., Kr\"atzel, E., K\"uhleitner, M., and Nowak, W. G.: Lattice
points in large regions and related arithmetic functions: recent
developments in a very classic topic, \emph{Elementare und analytische
Zahlentheorie}, 89--128, Schr. Wiss. Ges. Johann Wolfgang Goethe
Univ. Frankfurt am Main, 20, Franz Steiner Verlag Stuttgart,
Stuttgart, 2006.


\bibitem{koldobsky05book}
Koldobsky, A.: Fourier analysis in convex geometry, Mathematical Surveys and Monographs, 116. American Mathematical Society, Providence, RI, 2005.



\bibitem{LL16}
Laugesen, R. S., Liu, S.: Optimal stretching for lattice points and eigenvalues. ArXiv: 1609.06172. To appear in \emph{Ark. Mat}.

\bibitem{LL17}
Laugesen, R. S., Liu, S.: Shifted lattices and asymptotically optimal ellipses. ArXiv: 1707.08590.


\bibitem{marshall}
Marshall, N. F.: Stretching convex domains to capture many lattice points. ArXiv: 1707.00682.


\bibitem{MS}
Marshall, N. F., Steinerberger, S.: Triangles capturing many lattice points. ArXiv: 1706.04170.


\bibitem{mullerI}
M\"uller, W.: On the average order of the lattice rest of a convex body, \emph{Acta Arith}, \textbf{80}, 89--100 (1997).

\bibitem{nowak14survey}
Nowak, W. G.: Integer points in large bodies. \emph{Topics in mathematical analysis and applications}, 583–-599,
Springer Optim. Appl., \textbf{94}, Springer, Cham (2014).

\end{thebibliography}
\end{document}